\documentclass[a4paper,12pt]{article}

\usepackage{mymacros}

\newcommand{\NS}{\operatorname{NS}}
\newcommand{\Eff}{\operatorname{Eff}}

\newcommand{\ch}{\operatorname{ch}}
\newcommand{\td}{\operatorname{td}}
\newcommand{\Euler}{\operatorname{Euler}}
\newcommand{\Gm}{\bG_{\mathrm{m}}}

\newcommand{\bp}{\bsp}

\newcommand{\sing}{{\mathrm{sing}}}
\newcommand{\That}{{\widehat{T}}}

\newcommand{\HA}{\scrH_A}
\newcommand{\HB}{\scrH_B}
\newcommand{\HAZ}{H_{A, \, \bZ}}
\newcommand{\HBZ}{H_{B, \, \bZ}}

\newcommand{\HAC}{H_{A, \, \bC}}

\newcommand{\av}{{\Check{a}}}
\newcommand{\bv}{{\Check{b}}}
\newcommand{\cv}{{\Check{c}}}

\newcommand{\fv}{{\Check{f}}}

\newcommand{\hv}{{\Check{h}}}

\newcommand{\Xv}{{\Check{X}}}

\newcommand{\Yv}{{\Check{Y}}}
\newcommand{\Uv}{{\Check{U}}}

\newcommand{\scMv}{{\Check{\scM}}}
\newcommand{\scYv}{{\Check{\scY}}}
\newcommand{\frakYv}{{\Check{\frakY}}}
\newcommand{\varphiv}{{\Check{\varphi}}}
\newcommand{\bTv}{{\Check{\bT}}}
\newcommand{\Deltav}{{\Check{\Delta}}}
\newcommand{\Sigmav}{{\Check{\Sigma}}}

\newcommand{\bsgammav}{{\Check{\bsgamma}}}
\newcommand{\Mir}{\operatorname{Mir}}
\newcommand{\amb}{{\mathrm{amb}}}

\newcommand{\vc}{{\mathrm{vc}}}

\title{A note on exceptional unimodal singularities
and K3 surfaces}
\author{Masanori Kobayashi, Makiko Mase, and Kazushi Ueda}
\date{}
\begin{document}

\maketitle

\begin{abstract}
We study the relation
between the graded stable derived categories
of 14 exceptional unimodal singularities
and the derived categories of K3 surfaces
obtained as compactifications
of the Milnor fibers.
As a corollary,
we obtain a basis of the numerical Grothendieck group
similar to the one
given by Ebeling and Ploog
\cite{Ebeling-Ploog_MCPS}.
%We also discuss Hodge-theoretic mirror symmetry
%for these K3 surfaces.
\end{abstract}

\section{Introduction}
 \label{sc:introduction}

Let $f \in \bC[x, y, z]$
be a weighted homogeneous polynomial
defining one of {\em Arnold's 14 exceptional unimodal singularities}
\cite{Arnold_ICM}.
The list of corresponding weight systems
$(a, b, c; h) = \deg (x, y, z; f)$
is given in Table \ref{tb:exceptional}.
The quotient ring
$
 R = \bC[x,y,z] / (f)
$
%for a general choice of $f$
is the homogeneous coordinate ring
of a weighted projective line $\bX$
in the sense of Geigle and Lenzing
\cite{Geigle-Lenzing_WPC, Lenzing_WCARAF},
and the {\em Dolgachev number}
$\bsdelta = (\delta_1, \delta_2, \delta_3)$
of the singularity is defined
as the orders
%$\bp = (p_1, p_2, p_3)$
of the isotropy groups of $\bX$.
%\cite[Section 5]{Lenzing_WCARAF}.
\begin{table}
$$
\begin{array}{c|c|c|c|c}
 \text{name} & (a, b, c; h) &
 \bsdelta & \bsgamma & \text{dual} \\
  \hline
 E_{12} & (6,14,21;42) & (2,3,7) & (2,3,7) & E_{12} \\
  \hline
 E_{13} & (4,10,15;30) & (2,4,5) & (2,3,8) & Z_{11} \\
  \hline
 Z_{11} & (6,8,15;30) & (2,3,8) & (2,4,5) & E_{13} \\
  \hline
 E_{14} & (3,8,12;24) & (3,3,4) & (2,3,9) & Q_{10} \\
  \hline
 Q_{10} & (6,8,9;24) & (2,3,9) & (3,3,4) & E_{14} \\
  \hline
 Z_{12} & (4,6,11;22) & (2,4,6) & (2,4,6) & Z_{12} \\
  \hline
 W_{12} & (4,5,10;20) & (2,5,5) & (2,5,5) & W_{12} \\
  \hline
 Z_{13} & (3,5,9,18) & (3,3,5) & (2,4,7) & Q_{11} \\
  \hline
 Q_{11} & (4,6,7;18) & (2,4,7) & (3,3,5) & Z_{13} \\
  \hline
 W_{13} & (3,4,8,16) & (3,4,4) & (2,5,6) & S_{11} \\
  \hline
 S_{11} & (4,5,6,16) & (2,5,6) & (3,4,4) & W_{13} \\
  \hline
 Q_{12} & (3,5,6;15) & (3,3,6) & (3,3,6) & Q_{12} \\
  \hline
 S_{12} & (3,4,5;13) & (3,4,5) & (3,4,5) & S_{12} \\
  \hline
 U_{12} & (3,4,4;12) & (4,4,4) & (4,4,4) & U_{12}
\end{array}
$$
\caption{14 exceptional unimodal singularities}
\label{tb:exceptional}
\end{table}

On the other hand,
one can choose a distinguished basis
$(\alpha_i)_{i=1}^{\gamma_1 + \gamma_2 + \gamma_3}$
of vanishing cycles of $f$
so that the Coxeter--Dynkin diagram
is given by the diagram $\That(\gamma_1, \gamma_2, \gamma_3)$
shown in Figure \ref{fg:Dynkin-Milnor}.
The triple $\bsgamma=(\gamma_1, \gamma_2, \gamma_3)$
defined in this way is called
the {\em Gabrielov number} of the singularity.
\begin{figure}
\centering
\ifx\JPicScale\undefined\def\JPicScale{1}\fi
\psset{unit=\JPicScale mm}
\psset{linewidth=0.3,dotsep=1,hatchwidth=0.3,hatchsep=1.5,shadowsize=1,dimen=middle}
\psset{dotsize=0.7 2.5,dotscale=1 1,fillcolor=black}
\psset{arrowsize=1 2,arrowlength=1,arrowinset=0.25,tbarsize=0.7 5,bracketlength=0.15,rbracketlength=0.15}
\begin{pspicture}(0,0)(111,76)
\rput{0}(10,40){\psellipse[fillstyle=solid](0,0)(1,-1)}
\psline[linestyle=dotted](10,40)(45,40)
\psline(10,40)(35,40)
\psline(40,40)(65,40)
\psline[linestyle=dotted](85,60)(78,53)
\psline(65,40)(80,55)
\psline(85,60)(100,75)
\psline[linestyle=dotted](85,20)(78,27)
\psline(65,40)(80,25)
\psline(85,20)(100,5)
\psline(75,50)(65,55)
\psline(50,40)(65,55)
\psline(75,30)(65,55)
\psline(65,55)(65,70)
\psline[linestyle=dotted](66,40)(66,55)
\psline[linestyle=dotted](64,40)(64,55)
\rput{0}(25,40){\psellipse[fillstyle=solid](0,0)(1,-1)}
\rput{0}(50,40){\psellipse[fillstyle=solid](0,0)(1,-1)}
\rput{0}(65,40){\psellipse[fillstyle=solid](0,0)(1,-1)}
\rput{0}(75,50){\psellipse[fillstyle=solid](0,0)(1,-1)}
\rput{0}(90,65){\psellipse[fillstyle=solid](0,0)(1,-1)}
\rput{0}(65,55){\psellipse[fillstyle=solid](0,0)(1,-1)}
\rput{0}(65,70){\psellipse[fillstyle=solid](0,0)(1,-1)}
\rput{0}(75,30){\psellipse[fillstyle=solid](0,0)(1,-1)}
\rput{0}(100,75){\psellipse[fillstyle=solid](0,0)(1,-1)}
\rput{0}(90,15){\psellipse[fillstyle=solid](0,0)(1,-1)}
\rput{0}(100,5){\psellipse[fillstyle=solid](0,0)(1,-1)}
\rput(65,35){$\alpha_1$}
\rput(50,35){$\alpha_2$}
\rput(10,35){$\alpha_{\gamma_1}$}
\rput(82,33){$\alpha_{\gamma_1+1}$}
\rput(110,9){$\alpha_{\gamma_1+\gamma_2-1}$}
\rput(81,46){$\alpha_{\gamma_1+\gamma_2}$}
\rput(111,70){$\alpha_{\gamma_1+\gamma_2+\gamma_3-2}$}
\rput(53,55){$\alpha_{\gamma_1+\gamma_2+\gamma_3-1}$}
\rput(53,70){$\alpha_{\gamma_1+\gamma_2+\gamma_3}$}
\rput(25,35){$\alpha_{\gamma_1-1}$}
\rput(98,18){$\alpha_{\gamma_1+\gamma_2-2}$}
\rput(102,60){$\alpha_{\gamma_1+\gamma_2+\gamma_3-3}$}
\end{pspicture}
\caption{The diagram $\That(\gamma_1, \gamma_2, \gamma_3)$}
\label{fg:Dynkin-Milnor}
\end{figure}
The {\em strange duality}
discovered by Arnold \cite{Arnold_ICM}
states that
the 14 exceptional unimodal singularities come in pairs
$(f, \fv)$ such that
the Dolgachev number of $f$ is equal
to the Gabrielov number of $\fv$
and vice versa.
%Here $f$ can be equal to $\fv$,
%in which case $f$ is said to be {\em self-dual}.
Pinkham \cite{Pinkham_strange-duality} and
Dolgachev and Nikulin
\cite{Dolgachev_IQF, Nikulin_ISBF}
gave an interpretation
of the strange duality in terms of algebraic cycles
and transcendental cycles of K3 surfaces.

Let
$
 g(x,y,z,w) \in \bC[x, y, z, w]
$
be a very general weighted homogeneous polynomial
with
$
 \deg(x, y, z, w; g) = (a, b, c, 1; h)
$
and
$
 S = \bC[x,y,z,w]/(g)
$
be the quotient ring.
The Deligne--Mumford stack
$$
 \scY
  = \bProj S
  = [(g^{-1}(0) \setminus \bszero) / \bCx]
$$
is a compactification
of the Milnor fiber of $f$.
Here, the symbol $\bProj$ is used
to indicate that $\bProj S$ is considered
not as a scheme but as a stack.

The {\em stable derived category} of $S$
is defined as the quotient category
$$
 D^b_\sing(\gr S) = D^b(\gr S) / D^\perf(\gr S)
$$
of the bounded derived category
$
 D^b (\gr S)
$
of finitely generated $\bZ$-graded $S$-modules
by the full triangulated subcategory
$
 D^\perf(\gr S)
$
consisting of bounded complexes of projective modules
\cite{Buchweitz_MCM, Happel_GA, Krause_SDCNS, Orlov_DCCSTCS}.
Since $S$ is Gorenstein with parameter zero,
one has an equivalence
$$
 \Psi_S : D^b_\sing(\gr S) \simto D^b \coh \scY
$$
by Orlov \cite[Theorem 2.5]{Orlov_DCCSTCS}.
The stable derived category of $R=S/(w)$ is defined similarly as
$
 D^b_\sing(\gr R) = D^b(\gr R) / D^\perf(\gr R),
$
and studied by Kajiura, Saito and Takahashi
\cite{Kajiura-Saito-Takahashi_3}
and Lenzing and de la Pe\~{n}a
\cite{Lenzing-de_la_Pena_ECA}.
Since $R$ is Gorenstein with parameter $-1$,
one has a fully faithful functor
$$
 \Psi_R : D^b \coh \scY_\infty \to D^b_\sing(\gr R)
$$
and a semiorthogonal decomposition
$$
 D^b_\sing(\gr R) = \la \Psi_R(D^b \coh \scY_\infty), \, R/\frakm_R \ra
$$
by a result of Orlov \cite{Orlov_DCCSTCS},
where $R / \frakm_R$ is the residue field
by the maximal ideal $\frakm_R = (x, y, z)$
of the origin and
$
 \scY_\infty := \bProj R
$
is the divisor at infinity.

Let
$
 \Phi_{\gr} : \gr R \to \gr S
$
be the functor
sending an $R$-module
to the same module
considered as an $S$-module by the natural projection
$
 \varphi : S \to R.
$
Since $R$ is perfect as an $S$-module,
the functor
$
 \Phi_{\gr} % : D^b(\gr R) \to D^b(\gr S)
$
sends a perfect complex of $R$-modules
to a perfect complex of $S$-modules
and induces the {\em push-forward functor}
$$
 \Phi_\sing : D^b_\sing(\gr R) \to D^b_\sing(\gr S)
$$
studied in \cite{Dyckerhoff-Murfet_PFMF,
Polishchuk-Vaintrob_MFSCS}.
Let further
$
 \iota : \scY_\infty \hookrightarrow \scY
$
be the inclusion.
%Recall that $\scY_\infty$ is isomorphic
%to $\bX_\bsdelta$.

\begin{theorem} \label{th:functor}
The composite functor
$$
 \Psi_S \circ \Phi_\sing \circ \Psi_R
  : D^b \coh \scY_\infty \to D^b \coh \scY
$$
is isomorphic to the push-forward functor
$$
 \iota_* : D^b \coh \scY_\infty \to D^b \coh \scY,
$$
and the image of the residue field $R/\frakm_R$ in $D^b_\sing(\gr R)$
by $\Psi_S \circ \Phi_\sing$ is isomorphic to the structure sheaf
$\scO_\scY[2]$
shifted by $2$.
\end{theorem}

Let $Y$ be the minimal resolution
of the coarse moduli space of $\scY$.
The McKay correspondence as a derived equivalence
\cite{Kapranov-Vasserot, Bridgeland-King-Reid}
gives
\begin{equation} \label{eq:YvsX}
 \Upsilon : D^b \coh \scY \simto D^b \coh Y.
\end{equation}
Recall that the {\em numerical Grothendieck group}
$\scN(Y)$
is the quotient of the Grothendieck group $K(Y)$
of $Y$ by the radical of the Euler form
$$
 \chi \lb [\scE], [\scF] \rb
  = \sum_i (-1)^i \dim \Ext^i(\scE, \scF).
$$
The integral cohomology ring
$$
 H^\bullet (Y, \bZ)
  = H^0(Y, \bZ) \oplus H^2(Y, \bZ) \oplus H^4(Y, \bZ)
$$
equipped with the {\em Mukai pairing}
$$
 ((a_0, a_2, a_4), \ (b_0, b_2, b_4))
  = (a_2, b_2) - (a_0, b_4) - (a_4, b_0)
$$
is called the {\em Mukai lattice}.
For a coherent sheaf $\scE$,
its {\em Mukai vector} is defined by
$$
 v(\scE) = \ch(\scE) \sqrt{\td(Y)}.
$$
Riemann--Roch theorem states that
$$
 \chi(\scE, \scF) = - (v(\scE), v(\scF)),
$$
so that
$\scN(Y)$ can be identified with the image of $K(Y)$
in the Mukai lattice $H^\bullet(Y, \bZ)$
under the map
$
 v : K(Y) \to H^\bullet(Y, \bZ).
$
%$$
%\begin{array}{cccc}
% v : & \scN(Y) & \to & H^\bullet(Y, \bZ), \\
% & \vin & & \vin \\
% & [\scE] & \mapsto & \ch(\scE) \sqrt{\td(X)}.
%\end{array}
%$$

\begin{corollary} \label{cr:spherical}
There is a full exceptional collection
$(\scS_\alpha)_{\alpha=0}^{\delta_1+\delta_2+\delta_3-1}$
in $D^b_\sing(\gr R)$
the images
$
 \scE_\alpha = \Upsilon \circ \Psi_S \circ \Phi_\sing (\scS_\alpha)
$
of which satisfy the following:
\begin{itemize}
 \item
The endomorphism dg algebra of
$
 \bigoplus_{\alpha=1}^{\delta_1+\delta_2+\delta_3-1}
  \scE_\alpha
$
is the trivial extension
of the endomorphism dg algebra of
$
 \bigoplus_{\alpha=1}^{\delta_1+\delta_2+\delta_3-1}
  \scS_\alpha.
$
 \item
The sequence
$
 (\scE_\alpha)_{\alpha=0}^{\delta_1+\delta_2+\delta_3-1}
$
is a spherical collection.
 \item
The Coxeter--Dynkin diagram of the spherical collection
is $\That(\delta_1, \delta_2, \delta_3)$.
 \item
The spherical collection
is a basis of the numerical Grothendieck group $\scN(Y)$.
 \item
The spherical collection
split-generates $D^b \coh Y$.
\end{itemize}
\end{corollary}

Recall that an object $\scE$ is said to be {\em spherical}
if $\Hom^i(\scE, \scE)$ is isomorphic to $\bC$
for $i = 0, 2$ and zero otherwise
\cite[Definition 1.1]{Seidel-Thomas}.
A sequence of objects is called a {\em spherical collection}
if each object is spherical.
The definition of the {\em trivial extension}
can be found in \cite{Seidel_suspension},
which is called the {\em cyclic completion}
in \cite{Segal_ADT}.
The endomorphism dg algebra of
$
 \bigoplus_{\alpha=0}^{\delta_1+\delta_2+\delta_3-1}
  \scE_\alpha
$
is not the trivial extension
of the endomorphism dg algebra of
$
 \bigoplus_{\alpha=0}^{\delta_1+\delta_2+\delta_3-1}
  \scS_\alpha;
$
otherwise,
the derived category $D^b \coh Y$ will not depend
on the defining equation of $Y$.
The spherical collection
$
 (\scE_\alpha)_{\alpha=0}^{\delta_1+\delta_2+\delta_3-1}
$
has the same properties as the collection
given by Ebeling and Ploog
\cite{Ebeling-Ploog_MCPS}.

%We also discuss Hodge-theoretic aspect
%of mirror symmetry for K3 surfaces
%\cite{Aspinwall-Morrison_STKS,
%Dolgachev_MSK3,
%Morrison_MAMS,
%%Cox-Katz,
%Katzarkov-Kontsevich-Pantev,
%Iritani_QCP}.

Let $(Y, \Yv)$ be a pair of K3 surfaces
obtained as compactifications of Milnor fibers
of a dual pair of exceptional unimodal singularities.
According to \cite[Theorem 4.3.9]{Kobayashi_DW},
such a pair can be realized
as smooth anticanonical hypersurfaces
in a pair $(X, \Xv)$
of toric weak Fano manifolds
associated with a polar dual pair
$(\Delta, \Delta^*)$ of reflexive polytopes.
%Moreover,
%one can choose $X$ in such a way
%that the inclusion $\iota : Y \hookrightarrow X$
%induces an isomorphism
%$
% \iota^* : \NS(X) \to \NS(Y)
%$
%of the N\'{e}ron-Severi group
%for very general $Y$
%\cite[Theorem 4.3.9.(6)]{Kobayashi_DW}.

The {\em A-model VHS}
$
 (\HAZ, \nabla^A, \scrF_A^\bullet, Q_A)
$
associated with $Y$
is an integral variation of pure and polarized Hodge structures
of weight 2
in a neighborhood
$$
 U = \lc \beta + \sqrt{-1} \omega \in \NS(Y) \otimes \bC
       \bigm| \la \omega, d \ra \gg 0
           \text{ for any non-zero $d \in \Eff(Y)$}
     \rc
$$
of the large radius limit
in the complexified K\"{a}hler moduli space of $Y$.
Here
$\NS(Y) \subset \scN(Y)$ is the N\'{e}ron--Severi group of $Y$,
$\Eff(Y) \subset \scN(Y)$ is the semigroup of effective curves on $Y$,
$\HAZ$ is
the trivial local system on $U$
with fiber $\scN(Y)$ and
$\nabla^A = d$ is the associated trivial flat connection
on $\HA = \HAZ \otimes \scO_U$.
The polarization $Q_A$ is given by the Mukai pairing,
and the Hodge filtration is such that
$$
 \mho
  = \exp(\beta + \sqrt{-1} \omega)
  = \lb 1, \ (\beta + \sqrt{-1} \omega), \ 
      \frac{1}{2} (\beta + \sqrt{-1} \omega)^2 \rb
$$
spans the $(2, 0)$-part of $\HAC = \HAZ \otimes \bC$.

The K3 surface $\Yv$ comes in a family $\varphiv : \frakYv \to \scMv$
where $\scMv$ is an algebraic torus of the same dimension as $U$.
Let $\Uv'$ be a neighborhood of the {\em large complex structure limit}
in $\scMv$ and $\Uv$ be its universal cover.
The local system 
$
 R^2 \varphiv_! \, \bZ_\frakYv
$
carries an integral variation of polarized mixed Hodge structures,
and the {\em B-model VHS}
$
 \HBZ
$
is defined as the pull-back to $\Uv$
of the graded subquotient
$
 \gr_2^W \! R^2 \varphiv_! \, \bZ_\frakYv
$
of weight $2$.
There is a biholomorphic map
$
 \varsigma : \Uv \to U
$
called the {\em mirror map}.
Iritani has introduced
certain subsystems
$
 \HAZ^\amb \subset \HAZ
$
and
$
 \HBZ^\vc \subset \HBZ
$
and given an isomorphism
\begin{equation} \label{eq:ambient_isom}
 \Mir_\scY : \varsigma^*
  (\HAZ^\amb, \nabla^A, \scrF_A^\bullet, Q_A)
%    /\iota^*H^2(\scX; \bZ)
  \simto (\HBZ^\vc, \nabla^B, \scrF_B^\bullet, Q_B)
\end{equation}
of integral variations of pure and polarized Hodge structures
\cite[Theorem 6.9]{Iritani_QCP}.

\begin{corollary} \label{cr:vhs}
One has
$
 \HAZ^\amb = \HAZ
$
and
$
 \HBZ^\vc = \HBZ,
$
so that the isomorphism \eqref{eq:ambient_isom} gives
an isomorphism
$$
 \Mir_\scY : \varsigma^*
  (\HAZ, \nabla^A, \scrF_A^\bullet, Q_A)
  \simto (\HBZ, \nabla^B, \scrF_B^\bullet, Q_B).
$$
of integral variations of pure and polarized Hodge structures.
\end{corollary}

The equalities
$
 \HAZ^\amb = \HAZ
$
and
$
 \HBZ^\vc = \HBZ
$
fail in general
for a K3 hypersurface in a smooth toric weak Fano variety.
A typical counter-example is the case
when $Y$ is the quartic surface in the projective space,
where the class of a point
does not belong to $\HAZ^\amb$
and only four times the class of a point does
(cf. \cite[Section 6.6]{Iritani_QCP}).
Hodge-theoretic mirror symmetry for the quartic surface
is studied in detail by Hartmann \cite{Hartmann_PMQK}.

The organization of this paper is as follows:
We prove Theorem \ref{th:functor2}
in Section \ref{sc:push-forward},
which is slightly more general
than Theorem \ref{th:functor}.
In Section \ref{sc:spherical},
we use an exceptional collection
given by Lenzing and de la Pen\~{a}
\cite{Lenzing-de_la_Pena_ECA}
to prove Corollary \ref{cr:spherical}.
Variations of Hodge structures
is discussed in Section \ref{sc:vhs}.

{\bf Acknowledgment}:
We thank Hiroshi Iritani for very helpful discussions.
We also thank the anonymous referee
for pointing out an error in the earlier version and
suggesting a number of improvements.
M.~K. is supported by Grant-in-Aid for Scientific Research
(No.21540045).
K.~U. is supported by Grant-in-Aid for Young Scientists
(No.20740037).

\section{Push-forward in stable derived categories}
 \label{sc:push-forward}

Let $k$ be a field and
$A = \bigoplus_{i \ge 0} A_i$
be a Noetherian graded $k$-algebra.
We assume that $A$ is {\em connected}
in the sense that $A_0 = k$, and
write the maximal ideal as
$\frakm_A = \bigoplus_{i \ge 1} A_i$.
The graded ring $A$ is said to be {\em Gorenstein}
if $A$ has a finite injective dimension $n$ and
$$
 \RHom_A(k, A) = k(a)[-n]
$$
for some integer $a$,
which is called the Gorenstein parameter of $A$.
If $A = k[x_1, \ldots, x_n] / (f)$
for $\deg (x_1, \ldots, x_n; f) = (a_1, \ldots, a_n; h)$,
then $A$ is Gorenstein
%of dimension $r - 1$
with parameter $a = a_1 + \cdots + a_n - h$.
%(cf. e.g. \cite[Excercise 21.16]{Eisenbud_CA}).

Let $\gr A$ be the abelian category of
finitely generated $\bZ$-graded right $A$-modules, and
$\tor A$ be the full subcategory
consisting of graded modules
which are finite-dimensional over $k$.
The quotient category $\gr A / \tor A$
will be denoted by $\qgr A$,
which is equivalent to the abelian category
of coherent sheaves on the quotient stack
$\bProj A = [(\Spec A \setminus \bszero) / \Gm]$
by Serre's theorem
\cite[Proposition 2.16]{Orlov_DCCSTCS}.

Let $D^b(\gr A)$ be the bounded derived category
of $\gr A$.
An object of $D^b(\gr A)$ is said to be {\em perfect}
if it is quasi-isomorphic to a bounded complex
of projective modules.
The full subcategory of $D^b(\gr A)$
consisting of perfect complexes will be denoted
by $D^\perf(\gr A)$.
The quotient category
$$
 D^b_\sing(\gr A) = D^b(\gr A) / D^\perf(\gr A)
$$
is called the bounded {\em stable derived category} of $\gr A$
\cite{Buchweitz_MCM, Happel_GA, Krause_SDCNS, Orlov_TCS}.

Let $\scD$ be a $k$-linear triangulated category and
$\scN \subset \scD$ be a full triangulated subcategory.
The {\em right orthogonal} to $\scN$ is the full subcategory
$\scN^\perp \subset \scD$
consisting of objects $M$ satisfying
$\Hom(N, M) = 0$ for any $N \in \scN$.
The {\em left orthogonal} $\! \,^{\perp} \scN$ is defined similarly.
The subcategory $\scN$ is said to be {\em right admissible}
if the embedding $I : \scN \hookrightarrow \scD$ has
a right adjoint functor $Q : \scD \to \scN$.
Left admissibility is defined similarly
as the existence of a left adjoint functor,
and $\scN$ is said to be {\em admissible}
if it is both right and left admissible.
A subcategory $\scN$ is right admissible
if and only if for any $X \in \scD$,
there exists a distinguished triangle
$
 N \to X \to M \to N[1]
$
with $N \in \scN$ and $M \in \scN^\perp$.
Such a triangle is unique up to isomorphism,
and one has $Q(X) = N$ in this case.
If $\scN$ is right admissible,
then the quotient category $\scD / \scN$ is equivalent to $\scN^\perp$.
Analogous statements also hold for left admissible subcategories.
A sequence $(\scN_1, \ldots, \scN_n)$
of triangulated subcategories
in a triangulated category $\scD$ is called
a {\em weak semiorthogonal decomposition}
if there is a sequence
$
 \scD_1 = \scN_1 \subset \scD_2
  \subset \cdots \subset \scD_n = \scD
$
of left admissible subcategories
such that $\scN_i$ is left orthogonal to $\scD_{i-1}$
in $\scD_i$.
A weak semiorthogonal decomposition will be denoted by
$$
 \scD = \la \scN_1, \ldots, \scN_n \ra.
$$

An object $E$ of $\scD$ is {\em exceptional}
if $\Ext^i(E, E) = 0$ for $i \ne 0$ and
$\Hom(E, E)$ is spanned by the identity morphism.
An {\em exceptional collection} is
a sequence $(E_1, \dots, E_n)$ of exceptional objects
such that $\Ext^i(E_j, E_\ell) = 0$ for any $i$
and any $1 \le \ell < j \le n$.
A full triangulated subcategory
generated by an exceptional collection
is always admissible
\cite[Theorem 3.2]{Bondal_RAACS}.

For an integer $i$,
let $\gr A_{\ge i}$ be the full abelian subcategory of $\gr A$
consisting of graded modules $M$ such that $M_e = 0$ for $e < i$.
Let further $\scS_{\ge i}^A$ and $\scP_{\ge i}^A$
be the full triangulated subcategories of $D^b (\gr A)$
generated by graded torsion modules
$A/\frakm_A(e)$ for $e \le -i$ and
graded free modules $A(e)$ for $e \le -i$
respectively.
By \cite[Lemma 2.4]{Orlov_DCCSTCS},
the subcategories $\scS_{\ge i}^A$ and $\scP_{\ge i}^A$
are right and left admissible, respectively, in $D^b \gr A_{\ge i}$,
and let $\scD_i^A$ and $\scT_i^A$ be their right and left orthogonal
subcategories.
It follows that one has weak semiorthogonal decompositions
\begin{align}
 D^b(\gr A_{\ge i}) = \la \scD_i^A, \scS_{\ge i}^A \ra,
  \label{eq:DS} \\
 D^b(\gr A_{\ge i}) = \la \scP_{\ge i}^A, \scT_i^A \ra,
  \label{eq:PT}
\end{align} 
where $\scD_i^A$ is equivalent to the quotient category
$D^b(\gr A_{\ge i}) / \scS_{\ge i}^A$
which in turn is equivalent to $D^b(\qgr A)$,
and $\scT_i^A$ is equivalent to the quotient category
$D^b(\gr A_{\ge i}) / \scP_{\ge i}^A$
which in turn is equivalent to $D^b_\sing(\gr A)$.
In addition,
one has a semiorthogonal decomposition
\begin{align} \label{eq:TD}
 \scT_0^A = \la A/\frakm_A, A/\frakm_A(-1), \ldots,
  A/\frakm_A(a+1), \scD_{-a}^A \ra
\end{align}
if $a \le 0$
by Orlov \cite[Equation (12)]{Orlov_DCCSTCS}.

The semiorthogonal decomposition \eqref{eq:TD}
can be rephrased as
\begin{align} \label{eq:TD2}
 \scT_0^A = \la \scD_0^A, A/\frakm_A, A/\frakm_A(-1), \ldots,
  A/\frakm_A(a+1) \ra.
\end{align}
Indeed, one has the semiorthogonal decomposition
$$
 D^b(\gr A) = \la \scS_{<i}^A, D^b(\gr A_{\ge i}) \ra
$$
for any $i \in \bZ$
by \cite[Equation (7)]{Orlov_DCCSTCS},
which gives
$$
 D^b(\gr A_{\ge 0}) = \la A / \frakm_A, \dots, A / \frakm_A(a+1),
  D^b(\gr A_{\ge -a}) \ra.
$$
On the other hand,
one has
$$
 D^b(\gr A_{\ge -a}) = \la \scD_{-a}^A, \scS_{\ge -a}^A \ra
$$
%by \cite[Equation (9)]{Orlov_DCCSTCS},
by \eqref{eq:DS},
so that
\begin{align} \label{eq:SOD1}
 D^b(\gr A_{\ge 0}) =
  \la A / \frakm_A, \dots, A / \frakm_A(a+1), \scD_{-a}^A, \, \scS_{\ge -a}^A \ra.
\end{align}
By comparing
\begin{align*}
 D^b(\gr A_{\ge 0})
  = \la \scD_0^A, \, \scS_{\ge 0}^A \ra
  = \la \scD_0^A, A / \frakm_A, \dots, A / \frakm_A(a+1),
   \scS_{\ge -a}^A \ra
\end{align*}
with \eqref{eq:TD} and \eqref{eq:SOD1},
one obtains \eqref{eq:TD2}.

Let $\varphi : S \to R$ be a morphism
of graded connected Gorenstein rings,
and $\Phi_{\gr} : \gr R \to \gr S$ be the exact functor
which sends an $R$-module
to the same module considered as an $S$-module via $\varphi$.
%is exact, and we use the same notation
%for the induced functor
%$\varphi^* : D^b \gr R \to D^b \gr S$.
%on the bounded derived categories
The functor $\Phi_{\gr}$ sends finite-dimensional $R$-modules
to finite-dimensional $S$-modules,
and induces an exact functor
$\Phi_{\qgr} : \qgr R \to \qgr S$.

If $R$ has finite projective dimension as an $S$-module,
%(i.e. $R$ is perfect as an $S$-module).
then $\Phi_{\gr}$ sends perfect complexes of $R$-modules
to perfect complexes of $S$-modules,
and induces a functor
$
 \Phi_\sing : D^b_\sing(\gr R) \to D^b_\sing(\gr S)
$
of stable derived categories.

Now assume that $S$ is a graded connected Gorenstein ring
with parameter $a_S$, and
$w \in S_d$ is a homogeneous element of degree $d > 0$
which is not a zero divisor.
Then the exact sequence
$$
 0 \to S(-d) \xto{w} S \to R \to 0
$$
is a locally free resolution of the quotient ring $R = S / (w)$
as a graded $S$-module, and
$R$ is a graded connected Gorenstein ring
with parameter $a_R = a_S - d$
(see e.g. \cite[Proposition 2.2.10]{Goto-Watanabe_GRI}).
We write the natural projection as $\varphi : S \to R$,
which induces the functor $\Phi_{\gr} : \gr R \to \gr S$
as above.
In this case,
we have the following:

\begin{lemma} \label{lm:orthogonal}
If
%a bounded complex
$X \in D^b(\gr R)$
%of graded $R$-modules
satisfies
$
 \RHom_R (R / \frakm_R(i), X) = 0
$
for any $i \le 0$, then
%the same complex
$\Phi_{\gr}(X) \in D^b(\gr S)$
%considered as a complex of $S$-modules
satisfies
$
 \RHom_S(S / \frakm_S(i), \Phi_{\gr}(X)) = 0
$
for any $i \le 0$.
\end{lemma}

\begin{proof}
One has
\begin{align*}
 \RHom_S(S / \frakm_S(i), \Phi_{\gr}(X))
  &= \RHom_S(\Phi_{\gr}(R / \frakm_R(i)), \Phi_{\gr}(X)) \\
  &= \RHom_R((R / \frakm_R(i)) \Lotimes_S R, X) \\
  &= \RHom_R((R / \frakm_R(i)) \otimes_S \{ S(-d) \to S \}, X) \\
  &= \RHom_R( \{ R / \frakm_R(i-d) \to R / \frakm_R(i) \}, X),
\end{align*}
which vanishes since
$
 \RHom_R(R / \frakm_R(i-d), X)
  = \RHom_R(R / \frakm_R(i), X)
  = 0.
$
\end{proof}

Now we prove the following:

\begin{theorem} \label{th:functor2}
Let $S$ be a graded connected Gorenstein ring
with Gorenstein parameter $a_S = 0$,
and $R = S / (w)$ be the quotient ring
defined by a homogeneous element $w \in S_1$
of degree one
which is not a zero divisor.
Then the composite functor
$$
 D^b(\qgr R)
  \simto \scD_0^R
  \hookrightarrow \scT_0^R
  \simto D^b_\sing(\gr R)
  \xto{\Phi_{\sing}} D^b_\sing(\gr S)
  \simto \scT_0^S
  = \scD_0^S
  \simto D^b(\qgr S)
$$
is isomorphic to the functor
$\Phi_{\qgr} : D^b \qgr R \to D^b \qgr S$,
and the image of $R/\frakm_R \in \scT_0^R$ by
$$
 \scT_0^R
  \simto D^b_\sing(\gr R)
  \xto{\Phi_{\sing}} D^b_\sing(\gr S)
  \simto \scT_0^S
  = \scD_0^S
  \simto D^b(\qgr S)
$$
is isomorphic to $\scO[\dim S-1]$,
where $\scO$ is the image of the free module $S \in \gr S$
by the projection $\gr S \to \qgr S$ and
$[\dim S-1]$ is the shift in the derived category.
\end{theorem}

\begin{proof}
%To prove the first statement,
%recall that both the functors
%$
% \Phi_{\qgr} : D^b(\qgr R) \to D^b(\qgr S)
%$
%and
%$
% \Phi_{\sing} : D^b_\sing(\gr R) \to D^b_\sing(\gr S)
%$
%are induced from the push-forward
%$
% \Phi_{\gr} : D^b (\gr R) \to D^b (\gr S),
%$
%which sends a complex of $R$-modules
%to the same complex
%considered as a complex of $S$-modules
%by the natural surjection $\varphi : S \to R$.
The equivalence
$
 D^b(\qgr R) \simto \scD_0^R
$
is inverse to the composition
$$
 \scD_0^R
  \hookrightarrow
 \la \scD_0^R, \scS_{\ge 0}^R \ra
  =
 D^b(\gr R_{\ge 0})
  \to
 D^b(\gr R_{\ge 0}) / \scS_{\ge 0}^R
  \cong
 D^b(\gr R) / D^b(\tor R)
  =
 D^b(\qgr R),
$$
and the equivalence
$
 \scD_0^S \simto D^b(\qgr S)
$
is defined similarly.
Let $\Phi_\scD : \scD_0^R \to \scD_0^S$
be the functor
defined as the composition
$$
 \scD_0^R
  \hookrightarrow
 D^b(\gr R_{\ge 0})
  \xto{\Phi_{\gr}}
 D^b(\gr S_{\ge 0})
  \to
 \scD_0^S,
$$
where the last arrow is the left adjoint functor
to the embedding
$$
 \scD_0^S
  \hookrightarrow
 D^b(\gr S_{\ge 0})
  =
 \la \scD_0^S, \scS_{\ge 0}^S \ra.
$$
We show that
the following diagram commutes
up to natural isomorphism:
\begin{equation} \label{eq:comm1}
\begin{CD}
 D^b(\qgr R) @>{\Phi_{\qgr}}>> D^b(\qgr S) \\
 @V{\vsim}VV @AA{\vsim}A \\
 \scD_0^R @>{\Phi_\scD}>> \scD_0^S.
\end{CD}
\end{equation}
The quotient category
$D^b(\qgr R)=D^b(\gr R) / D^b(\tor R)$
has the same set of objects
as $D^b(\gr R)$,
and only morphisms are different.
The functor $\Phi_{\qgr}$ sends
an object $X \in D^b(\qgr R)$,
which is a complex of $R$-modules,
to the same complex,
considered as an object of $D^b(\qgr S)$.
The equivalence
$
 D^b(\gr R_{\ge 0})/\scS_{\ge 0}^R
  \cong D^b(\gr R)/D^b(\tor R)
$
allows one to assume that $X$ is an object
of $D^b(\gr R_{\ge 0})$.
The functor $D^b(\qgr R) \simto \scD_0^R$ sends
the object $X$ to an object $M \in \scD_0^R$
which fits in a distinguished triangle
\begin{equation} \label{eq:n1}
 N \to X \to M \to N[1]
\end{equation}
with $N \in \scS_{\ge 0}^R$.
The image $\Phi_\scD(M) \in \scD_0^S$
fits in the distinguished triangle
\begin{equation} \label{eq:n2}
 N' \to M \to \Phi_\scD(M) \to N'[1]
\end{equation}
where $M$ is considered
as an object of $D^b(\gr S_{\ge 0})$
and $N' \in \scS_{\ge 0}^S$.
Since both $N$ and $N'$ belong
to $D^b(\gr S_{\ge 0})$
as a complex of $S$-modules,
both $N$ and $N'$ are isomorphic
to the zero object in $D^b(\qgr S)$ and
exact sequences \eqref{eq:n1} and \eqref{eq:n2}
give the isomorphisms
$X \simto M \simto \Phi_{\scD}(M)$
in $D^b(\qgr S)$.
This isomorphism is natural
since all constructions above are functorial,
and the commutativity of \eqref{eq:comm1} is proved.

One can similarly define the functor
$
 \Phi_\scT : \scT_0^R \to \scT_0^S
$
as the composition
$$
 \scT_0^R
  \hookrightarrow
 \la \scP_{\ge 0}^R, \scT_0^R \ra
  =
 D^b (\gr R_{\ge 0})
  \xto{\Phi_{\gr}}
 D^b (\gr S_{\ge 0})
  \to
 \scT_0^S
$$
and prove that
the following diagram also commutes
up to natural isomorphism:
\begin{equation} \label{eq:comm2}
\begin{CD}
 \scT_0^R @>{\Phi_\scT}>> \scT_0^S \\
 @V{\vsim}VV @AA{\vsim}A \\
 D^b_\sing(\gr R) @>{\Phi_{\sing}}>> D^b_\sing(\gr S)
\end{CD}
\end{equation}
To prove the first statement
in Theorem \ref{th:functor2},
it remains to see that
the following diagram commutes,
where $I : \scD_0^R \hookrightarrow \scT_0^R$
is the embedding
coming from the semiorthogonal decomposition
in \eqref{eq:TD2}:
\begin{equation} \label{eq:comm3}
\begin{CD}
 \scD_0^R @>{I}>> \scT_0^R \\
 @V{\Phi_\scD}VV @VV{\Phi_\scT}V \\
 \scD_0^S @= \scT_0^S
\end{CD}
\end{equation}
This comes from the fact
that for any object $X \in \scD_0^R \subset D^b(\gr R_{\ge 0})$,
the object
$\Phi_{\gr}(X) \in D^b(\gr S_{\ge 0})$
is right orthogonal to $\scS_{\ge 0}^S$
by Lemma \ref{lm:orthogonal},
and hence belongs to $\scD_0^S = \scT_0^S$;
$$
 \Phi_{\scD}(X) = \Phi_{\scT} \circ I (X) = \Phi_{\gr}(X).
$$

For the second statement,
note that the image of $R/\frakm_R \in \scT_0^R$ by the composition
$$
 \scT_0^R
  \simto D^b_\sing(\gr R)
  \xto{\Phi_{\sing}} D^b_\sing(\gr S)
$$
is $S/\frakm_S$.
Its image by the equivalence
$$
 D^b_\sing(\gr S)
  \cong D^b(\gr S_{\ge 0}) / \scP_{\ge 0}^S
  \simto \scT_0^S
$$
is the object $N \in \scT_0^S$
which fits in the distinguished triangle
$$
 N \to S/\frakm_S \to M \to N[1]
$$
with $M \in \scP_{\ge 0}^S$.
Since $S$ is Gorenstein with parameter zero,
one has
$$
 \Hom(S/\frakm_S, S(i)) =
\begin{cases}
 k[- \dim S] & i = 0, \\
 0 & \text{otherwise},
\end{cases}
$$
which shows that the cone
$
 N = \Cone(S/\frakm_S[-1] \to S[\dim S - 1])
$
belongs to $\scT_0^S$ and
satisfies the desired property
with $M = S[\dim S]$.
It is clear that
$
 S/\frakm_S[-1] \in D^b(\gr A_{\ge 0})
$
goes to
$
 0 \in D^b(\qgr A)
$
and
$
 S[\dim S - 1] \in D^b(\gr A_{\ge 0})
$
goes to
$
 \scO[\dim S - 1] \in D^b(\qgr A),
$
so that
$
 N \in \scT_0^S = \scD_0^S \subset D^b(\gr A_{\ge 0})
$
goes to
$
 \scO[\dim S - 1] \in D^b(\qgr A),
$
and Theorem \ref{th:functor2} it proved.
\end{proof}

\section{Spherical collections on K3 surfaces}
 \label{sc:spherical}

Let $\bX$ be the weighted projective line
with weight $\bp = (p_1, p_2, p_3)$
in the sense of Geigle and Lenzing \cite{Geigle-Lenzing_WPC}.
The abelian category $\coh \bX$ of coherent sheaves on $\bX$
is equivalent by Serre's theorem
\cite[Section 1.8]{Geigle-Lenzing_WPC}
to the quotient category
$
% \coh \bX =
  \gr T / \tor T
$
of the abelian category $\gr T$
of finitely generated $L$-graded $T$-modules
by the full subcategory $\tor T$
consisting of torsion modules.
Here $L$ is the abelian group of rank one
generated by four elements
$\vecx_1$,
$\vecx_2$,
$\vecx_3$, and
$\vecc$
with relations
$
 p_1 \vecx_1 = p_2 \vecx_2 = p_3 \vecx_3 = \vecc,
$
and
$
 T = \bC[x_1, x_2, x_3] / (x_1^{p_1} + x_2^{p_2} + x_3^{p_3})
$
is an $L$-graded ring
of Krull dimension two.
%Serre duality on $\bX$ is given by
%$$
% \Ext^1(X, Y)^\vee \cong \Hom(Y, X(\vecomega)),
%$$
%where
%$
% \vecomega = \vecc - \vecx_1 - \vecx_2 - \vecx_3.
%$
Let
$$
 (\scP_\alpha)_{\alpha=1}^{p_1+p_2+p_3-1} =
  (\scO,
   U_1^{(1)}, \ldots, U_1^{(p_1-1)},
   U_2^1, \ldots, U_2^{(p_2-1)},
   U_3^1, \ldots, U_3^{(p_3-1)},
   \scO(-\vecomega-\vecc)[1])
$$
be the full strong exceptional collection
given by Lenzing and de la Pen\~{a}
\cite[Proposition 3.9]{Lenzing-de_la_Pena_ECA},
where $U_i^{(j)}$ are defined by
\begin{align*}
 U_i^{(j)} &=
  \coker( \scO(-(p_i-1) \vecx_i)
   \hookrightarrow \scO((-p_i+1+j) \vecx_i ),
\end{align*}
and
$
 \vecomega = \vecc - \vecx_1 - \vecx_2 - \vecx_3
  \in L
$
is the dualizing element
\cite[Theorem 2.2]{Geigle-Lenzing_WPC}.
Let
$
 (\scS_\beta)_{\beta=1}^{p_1+p_2+p_3-1}
$
be the right dual collection to
$
 (\scP_\alpha)_{\alpha=1}^{p_1+p_2+p_3-1},
$
which is characterized by the property
$$
 \dim \Hom(\scP_\alpha, \scS_\beta)
  = \delta_{\alpha, \, p_1+p_2+p_3-\beta},
$$
and given explicitly as
$$
 (\scS_\beta)_{\beta=1}^{p_1+p_2+p_3-1} =
  (\scO(-\vecc)[2],
    \scO(-\vecx_1)[1],
    S_1^{(p_1-2)}, \ldots, S_1^{(1)},
    \ldots,
    \scO(-\vecx_3)[1],
    S_3^{(p_3-2)}, \ldots, S_3^{(1)},
   \scO),
$$
where
$$
 S_i^{(j)} = \coker(\scO(-(p_i-j-2) \vecx_i)
  \hookrightarrow \scO(-(p_i-j-1) \vecx_i)).
$$
%The collection
%$
% (\scP_\alpha)_{\alpha=1}^{p_1+p_2+p_3-1}
%$
%is shown in Figures \ref{fg:strong_quiver}.
The total morphism algebra of the collection
$
 (\scP_\alpha)_{\alpha=1}^{p_1+p_2+p_3-1}
$
is isomorphic to the path algebra of the quiver
shown in Figure \ref{fg:strong_quiver},
where two dotted arrows represent two relations.
In terms of quiver representations,
$\scP_\alpha$ are projective modules and
$\scS_\alpha$ are simple modules,
and one has
$$
 \dim \Hom^i(\scS_\alpha, \scS_\beta) =
\begin{cases}
 \delta_{\alpha \beta} & i = 0, \\
 \#(\text{solid arrows from $\beta$ to $\alpha$})
   & i = 1, \\
 \#(\text{dotted arrows from $\beta$ to $\alpha$})
   & i = 2.
\end{cases}
$$

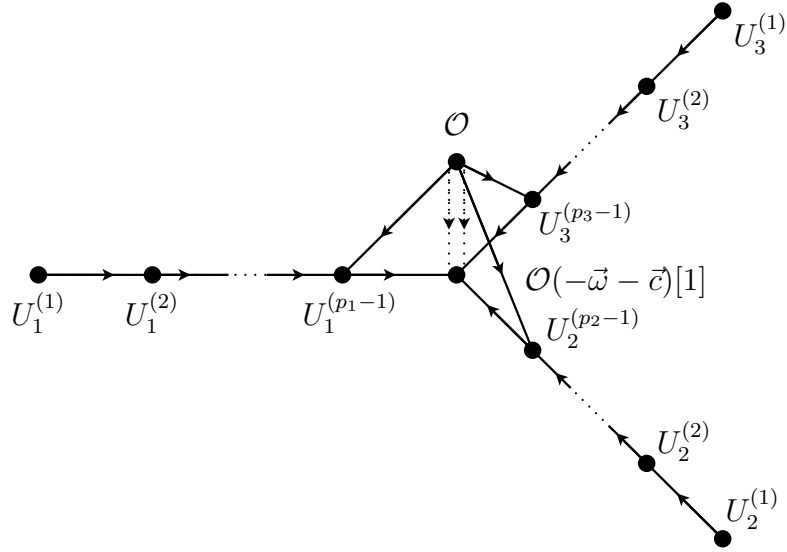
\begin{figure}
\centering
\ifx\JPicScale\undefined\def\JPicScale{1}\fi
\psset{unit=\JPicScale mm}
\psset{linewidth=0.3,dotsep=1,hatchwidth=0.3,hatchsep=1.5,shadowsize=1,dimen=middle}
\psset{dotsize=0.7 2.5,dotscale=1 1,fillcolor=black}
\psset{arrowsize=1 2,arrowlength=1,arrowinset=0.25,tbarsize=0.7 5,bracketlength=0.15,rbracketlength=0.15}
\begin{pspicture}(0,0)(105,76)
\rput{0}(10,40){\psellipse[fillstyle=solid](0,0)(1,-1)}
\psline[linestyle=dotted]{->}(10,40)(45,40)
\psline(10,40)(35,40)
\psline(40,40)(65,40)
\psline[linestyle=dotted]{->}(85,60)(78,53)
\psline(65,40)(80,55)
\psline(85,60)(100,75)
\psline[linestyle=dotted]{->}(85,20)(78,27)
\psline(65,40)(80,25)
\psline(85,20)(100,5)
\psline(75,50)(65,55)
\psline(50,40)(65,55)
\psline(75,30)(65,55)
\psline{->}(10,40)(20,40)
\psline{->}(25,40)(30,40)
\psline{->}(50,40)(57,40)
\psline{->}(75,30)(69,36)
\psline{->}(90,15)(86,19)
\psline{->}(100,5)(94,11)
\psline{->}(65,55)(70,52.5)
\psline{->}(65,55)(55,45)
\psline{->}(65,55)(71,40)
\psline{->}(75,50)(70,45)
\psline{->}(90,65)(86,61)
\psline{->}(100,75)(94,69)
\psline[linestyle=dotted](66,40)(66,55)
\psline[linestyle=dotted](64,40)(64,55)
\psline[linestyle=dotted]{->}(64,55)(64,46)
\psline[linestyle=dotted]{->}(66,55)(66,46)
\rput{0}(25,40){\psellipse[fillstyle=solid](0,0)(1,-1)}
\rput{0}(50,40){\psellipse[fillstyle=solid](0,0)(1,-1)}
\rput{0}(65,40){\psellipse[fillstyle=solid](0,0)(1,-1)}
\rput{0}(75,50){\psellipse[fillstyle=solid](0,0)(1,-1)}
\rput{0}(90,65){\psellipse[fillstyle=solid](0,0)(1,-1)}
\rput{0}(65,55){\psellipse[fillstyle=solid](0,0)(1,-1)}
\rput{0}(75,30){\psellipse[fillstyle=solid](0,0)(1,-1)}
\rput{0}(100,75){\psellipse[fillstyle=solid](0,0)(1,-1)}
\rput{0}(90,15){\psellipse[fillstyle=solid](0,0)(1,-1)}
\rput{0}(100,5){\psellipse[fillstyle=solid](0,0)(1,-1)}
\rput(10,35){$U_1^{(1)}$}
\rput(25,35){$U_1^{(2)}$}
\rput(51,35){$U_1^{(p_1-1)}$}
\rput(104,9){$U_2^{(1)}$}
\rput(95,18){$U_2^{(2)}$}
\rput(83,33){$U_2^{(p_2-1)}$}
\rput(65,60){$\scO$}
\rput(86,39){$\scO(-\vecomega-\vecc)[1]$}
\rput(105,72){$U_3^{(1)}$}
\rput(95,62){$U_3^{(2)}$}
\rput(82,47){$U_3^{(p_3-1)}$}
\end{pspicture}
\caption{The full strong exceptional collection on $\bX_\bp$}
\label{fg:strong_quiver}
\end{figure}

%\begin{figure}
%\centering
%\input{dual_quiver.pst}
%\caption{The dual exceptional collection on $\bX_\bp$}
%\label{fg:dual_quiver}
%\end{figure}

Let $\scK$ be the total space
of the canonical bundle of $\bX$.
Since the collection
$
 ( \scS_{\alpha} )_{\alpha=1}^{p_1+p_2+p_3-1}
$
is full,
the push-forward
$
 ( \iota_* \scS_{\alpha} )_{\alpha=1}^{p_1+p_2+p_3-1}
$
generates the derived category $D^b \coh_\bX \scK$
of coherent sheaves on $\scK$
supported on the image of the zero section
$\iota : \bX \to \scK$.

\begin{theorem}[{Segal \cite[Theorem 4.2]{Segal_ADT},
Ballard \cite[Proposition 4.14]{Ballard_SLCY}}]
Let $\scS$ be an object of $D^b \coh \bX$ and
$\iota_* \scS$ be the push-forward of $\scS$
along the zero-section.
Then the endomorphism dg algebra of $\iota_* \scS$
is the trivial extension of the endomorphism dg algebra of $\scS$.
\end{theorem}

It follows that
$$
 \Hom^i(\iota_* \scS_\alpha, \iota_* \scS_\beta)
  = \Hom^i(\scS_\alpha, \scS_\beta)
     \oplus \Hom^{2-i}(\scS_\beta, \scS_\alpha)^\vee,
$$
so that
\begin{equation} \label{eq:ab}
 \chi(\iota_* \scS_\alpha, \iota_* \scS_\beta) =
  \begin{cases}
   2 & \text{if $\alpha = \beta$}, \\
   -1 & \text{if $\alpha$ and $\beta$ are connected by a solid arrow}, \\
   2 & \text{if $\alpha$ and $\beta$ are connected by two dotted arrows}, \\
   0 & \text{otherwise}
 \end{cases}
\end{equation}
for $1 \le \alpha, \beta \le p_1+p_2+p_3-1$.

Let $\scY$ be a very general hypersurface of degree $h$
in $\bP(a, b, c, 1)$,
where $(a,b,c;h)$ is a weight system
in Table \ref{tb:exceptional}.
%associated with any of 14 exceptional unimodal singularities
%as in Section \ref{sc:introduction}.
The divisor
$
 \scY_\infty = \{ w = 0 \} \subset \scY
$
at infinity is a weighted projective line
whose weight is given by the Dolgachev number of the singularity;
$
 (p_1,p_2,p_3) = (\delta_1,\delta_2,\delta_3).
$
Note that 
the formal neighborhood of $\scY_\infty$ in $\scY$
is isomorphic to the formal neighborhood of $\bX$ in $\scK$;
it suffices
(see e.g. \cite[Theorem 1.6]{Camacho-Movasati}
%for a proof for schemes
)
to show
$
 H^1(\scT_\bX \otimes (\scN_{\bX / \scK}^\vee)^\nu) = 0
$
and
$
 H^1((\scN_{\bX / \scK}^\vee)^\nu) = 0
$
for any $\nu \ge 1$,
which easily follows from the fact that
both the tangent sheaf
$\scT_\bX$
and the conormal sheaf
$\scN_{\bX / \scK}^\vee$
are isomorphic to
$
% \scT_\bX
%  \cong N_{\bX / \scK}^\vee
%  \cong 
  \scO_\bX(- \vecomega).
$
We fix such an isomorphism,
which induces an equivalence
\begin{equation} \label{eq:KvsX}
 D^b \coh_{\bX} \scK \cong D^b \coh_{\scY_\infty} \scY
\end{equation}
of triangulated categories.
Since
$$
 \Hom^*(\scO_\scY, \iota_* \scS_\alpha)
  \cong H^* (\iota_* \scS_\alpha)
  \cong H^* (\scS_\alpha)
  \cong \Hom^* (\scS_{p_1+p_2+p_3-1}, \scS_\alpha),
$$
one has
\begin{equation} \label{eq:0a}
 \dim \Hom^i(\scO_\scY[1], \iota_* \scS_\alpha)
  = \delta_{i 1} \delta_{\alpha, \, p_1+p_2+p_3-1},
\end{equation}
so that the Euler form on the spherical collection
$
 (\scO_{\scY}[1], \iota_* \scS_1, \ldots,
  \iota_* \scS_{p_1+p_2+p_3-1})
$
is identical to the spherical collection
in Figure \ref{fg:EP_collection}
given by Ebeling and Ploog \cite{Ebeling-Ploog_MCPS}.

\begin{lemma} \label{lem:split-generate}
The spherical collection
$$
 (\scO_{\scY}, \iota_* \scS_1, \ldots,
   \iota_* \scS_{p_1+p_2+p_3-1})
$$
split-generates $D^b \coh \scY$.
\end{lemma}

\begin{proof}
The line bundle $\scO_{\scY}(- k \scY_\infty)$
is contained in the full triangulated subcategory
of $D^b \coh \scY$
generated by the above spherical collection
for any $k \in \bN$,
since the cokernel of the inclusion
$
 \scO_{\scY}(- k \scY_\infty) \hookrightarrow \scO_\scY
$
is supported on $\scY_\infty$
and hence contained in $\coh_{\scY_\infty} \scY$.
For any coherent sheaf $\scE$,
there is a surjection
$$
 \varphi_0 : \scO_\scY(- n_0 \scY_\infty)^{\oplus k_0} \to \scE
$$
for sufficiently large $n_0$ and $k_0$
(i.e. the hyperplane section $\scY_\infty$ is ample).
Let $\scE_1 = \ker \varphi_0$ be the kernel of this morphism.
Then there is a surjection
$$
 \varphi_1 : \scO_\scY(- n_1 \scY_\infty)^{\oplus k_1} \to \scE_1
$$
for sufficiently large $n_1$ and $k_1$,
and one can set $\scE_2 = \ker \varphi_1$.
By repeating this process,
one obtains a distinguished triangle
$$
 \scE_{k+1}[k] \to \scF \to \scE \xto{[+1]} \scE_{k+1}[k+1],
$$
where $\scE_{k+1}$ is a coherent sheaf and
$$
 \scF = \lc \scO_\scY(- n_k \scY_\infty)^{\oplus m_k}
  \xto{\varphi_{k}}
   \scO_\scY(- n_{k-1} \scY_\infty)^{\oplus m_{k-1}}
  \xto{\varphi_{k-1}} \cdots \xto{\varphi_0}
   \scO_\scY(-n_0 \scY_\infty)^{\oplus k_0} \rc
$$
for any $k \ge 0$.
Since $\scY$ is smooth,
the homological dimension of $\coh \scY$ is equal
to the dimension of $\scY$, and
this triangle splits for $k > \dim \scY$.
It follows that any coherent sheaf is a direct summand
of a complex of locally free sheaves
contained in the full triangulated subcategory of $D^b \coh \scY$
generated by $(\scS_\beta)_{\beta=0}^{p_1+p_2+p_3-1}$,
and Lemma \ref{lem:split-generate} is proved.
\end{proof}

%\section{A spherical collection on $Y$}

Let $Y$ be the minimal resolution
of the coarse moduli space of $\scY$.
It can be realized as an anticanonical K3 hypersurface
in a toric weak Fano manifold $X$
\cite{Kobayashi_DW}.
It contains the Milnor fiber as an open subset, and
the complement consists of chains of $(-2)$-curves
intersecting as in Figure \ref{fg:divisor_graph}.
It follows that the transcendental lattice of $Y$
is isomorphic to the Milnor lattice of $Y$.
By the McKay correspondence
as a derived equivalence
\cite{Kapranov-Vasserot, Bridgeland-King-Reid},
one has an equivalence
\begin{equation} \label{eq:derived_equiv}
 \Upsilon : D^b \coh \scY \simto D^b \coh Y
\end{equation}
of triangulated categories.
%The choice of the equivalence \eqref{eq:derived_equiv}
%is not unique, and
%the one given by Kapranov and Vasserot
%\cite[\S 2]{Kapranov-Vasserot} sends
%$\scO_{\scY_\infty}$ to $\scO_E$,
%$\scO_{\scY_\infty}(-1)$ to $\scO_E(-1)$,
%$\scS_i^{(p_i-1)}$ to $\scO_{E_i}$ and
%$\scS_i^{(j)}$ to $\scO_{E_{j+1}^i}(-1)[1]$
%for $j = 0, \ldots, p_i-2$
%where $E_i = E_1^i \cup \cdots \cup E_{p_i-1}^i$ and
%$E = E_\infty \cup E_1 \cup E_2 \cup E_3$.
Set $\scE_0 = \scO_Y[1]$ and
$\scE_\alpha = \Upsilon \circ \iota_* (\scS_\alpha)$
for $\alpha = 1, \ldots, p_1+p_2+p_3-1$.

\begin{proposition} \label{prop:NY}
The numerical Grothendieck group $\scN(Y)$ is spanned by
$
 ([\scE_\alpha])_{\alpha=0}^{p_1+p_2+p_3-1}
$
and isomorphic to the lattice $\That(p_1, p_2, p_3)$.
\end{proposition}

\begin{proof}
The numerical Grothendieck group $\scN(Y)$ is generated
by the class $[\scO_Y]$ of the structure sheaf,
the N\'{e}ron--Severi group $\NS(Y)$,
and the class $[\scO_p]$ of a skyscraper sheaf.
The structure of $\NS(Y)$ for very general $Y$ is well studied
(see e.g. \cite{Belcastro_PLFK3S}),
and generated by the irreducible components
of the divisor
$
 E = E_\infty \cup \bigcup_{i=1}^3 \bigcup_{j=1}^{p_i-1} E_j^i
$
at infinity.
Both the structure sheaves of irreducible components of $E$
and a skyscraper sheaf $\scO_p$ on $E$ belong to $D^b \coh_E Y$,
which is equivalent to $D^b \coh_{\scY_\infty} \scY$
by the functor $\Upsilon$.
Since $D^b \coh_{\scY_\infty} \scY$ is generated by
$
 (\iota_* \scS_\alpha)_{\alpha=1}^{p_1+p_2+p_3-1},
$
the collection
$
 ([\scE_\alpha])_{\alpha=1}^{p_1+p_2+p_3-1}
$
generates $\NS(Y)$ and $[\scO_p]$,
so that the collection
$
 ([\scE_\alpha])_{\alpha=0}^{p_1+p_2+p_3-1}
$
generates $\scN(Y)$.
Since $\rank \scN(Y) = p_1+p_2+p_3$,
the collection
$
 ([\scE_\alpha])_{\alpha=0}^{p_1+p_2+p_3-1}
$
is a basis of $\scN(Y)$.
It is clear from \eqref{eq:ab} and \eqref{eq:0a} that
$\scN(Y)$ is isomorphic to $\That(p_1, p_2, p_3)$
as a lattice, and
Proposition \ref{prop:NY} is proved.
\end{proof}

It is an interesting problem to see if the collection
$
 (\scE_\alpha)_{\alpha=0}^{p_1+p_2+p_3-1}
$
can be related to the collection of Ebeling and Ploog
\cite{Ebeling-Ploog_MCPS}
shown in Figure \ref{fg:EP_collection}
by an autoequivalence of $D^b \coh Y$.

\begin{figure}
\centering
\ifx\JPicScale\undefined\def\JPicScale{1}\fi
\psset{unit=\JPicScale mm}
\psset{linewidth=0.3,dotsep=1,hatchwidth=0.3,hatchsep=1.5,shadowsize=1,dimen=middle}
\psset{dotsize=0.7 2.5,dotscale=1 1,fillcolor=black}
\psset{arrowsize=1 2,arrowlength=1,arrowinset=0.25,tbarsize=0.7 5,bracketlength=0.15,rbracketlength=0.15}
\begin{pspicture}(0,0)(110,76)
\rput{0}(10,40){\psellipse[fillstyle=solid](0,0)(1,-1)}
\psline[linestyle=dotted](10,40)(45,40)
\psline(10,40)(35,40)
\psline(40,40)(65,40)
\psline[linestyle=dotted](85,60)(78,53)
\psline(65,40)(80,55)
\psline(85,60)(100,75)
\psline[linestyle=dotted](85,20)(78,27)
\psline(65,40)(80,25)
\psline(85,20)(100,5)
\rput{0}(25,40){\psellipse[fillstyle=solid](0,0)(1,-1)}
\rput{0}(50,40){\psellipse[fillstyle=solid](0,0)(1,-1)}
\rput{0}(65,40){\psellipse[fillstyle=solid](0,0)(1,-1)}
\rput{0}(75,50){\psellipse[fillstyle=solid](0,0)(1,-1)}
\rput{0}(90,65){\psellipse[fillstyle=solid](0,0)(1,-1)}
\rput{0}(75,30){\psellipse[fillstyle=solid](0,0)(1,-1)}
\rput{0}(100,75){\psellipse[fillstyle=solid](0,0)(1,-1)}
\rput{0}(90,15){\psellipse[fillstyle=solid](0,0)(1,-1)}
\rput{0}(100,5){\psellipse[fillstyle=solid](0,0)(1,-1)}
\rput(10,35){$E_1^1$}
\rput(25,35){$E_2^1$}
\rput(50,35){$E_{p_1-1}^1$}
\rput(110,5){$E_1^2$}
\rput(100,15){$E_2^2$}
\rput(85,30){$E_{p_2-1}^2$}
\rput(110,75){$E_1^3$}
\rput(100,65){$E_2^3$}
\rput(85,50){$E_{p_3-1}^3$}
\rput(75,40){$E_\infty$}
\rput(10,35){}
\end{pspicture}
\caption{The configuration of $(-2)$-curves at infinity}
\label{fg:divisor_graph}
\end{figure}
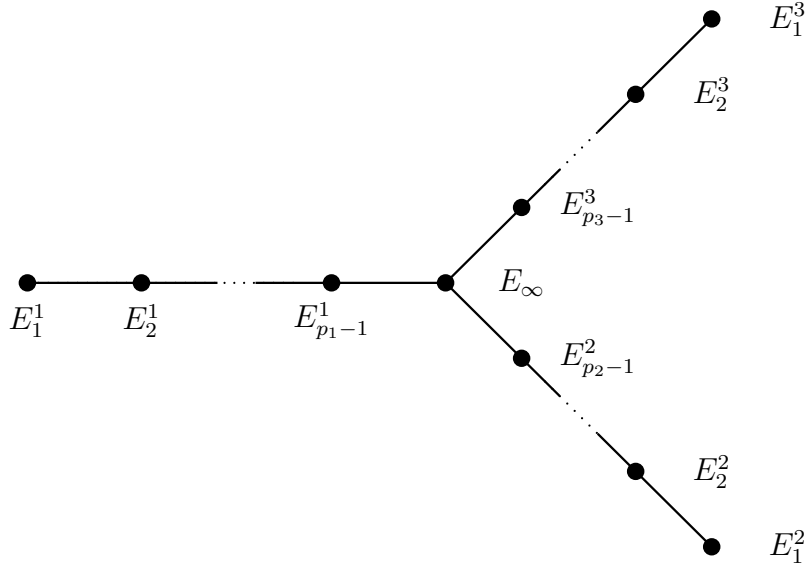

\begin{figure}
\centering
\ifx\JPicScale\undefined\def\JPicScale{1}\fi
\psset{unit=\JPicScale mm}
\psset{linewidth=0.3,dotsep=1,hatchwidth=0.3,hatchsep=1.5,shadowsize=1,dimen=middle}
\psset{dotsize=0.7 2.5,dotscale=1 1,fillcolor=black}
\psset{arrowsize=1 2,arrowlength=1,arrowinset=0.25,tbarsize=0.7 5,bracketlength=0.15,rbracketlength=0.15}
\begin{pspicture}(0,0)(111,76)
\rput(84,40){$\scO_{E_\infty}(-1)$}
\rput{0}(10,40){\psellipse[fillstyle=solid](0,0)(1,-1)}
\psline[linestyle=dotted](10,40)(45,40)
\psline(10,40)(35,40)
\psline(40,40)(65,40)
\psline[linestyle=dotted](85,60)(78,53)
\psline(65,40)(80,55)
\psline(85,60)(100,75)
\psline[linestyle=dotted](85,20)(78,27)
\psline(65,40)(80,25)
\psline(85,20)(100,5)
\psline(75,50)(65,55)
\psline(50,40)(65,55)
\psline(75,30)(65,55)
\psline(65,55)(65,70)
\psline[linestyle=dotted](66,40)(66,55)
\psline[linestyle=dotted](64,40)(64,55)
\rput(58,70){$\scO_Y$}
\rput(58,55){$\scO_{E_\infty}$}
\rput(8,35){$\scO_{E_1^1}(-1)$}
\rput(88,3){$\scO_{E_1^2}(-1)$}
\rput(63,29){$\scO_{E_{p_2-1}^2}(-1)$}
\rput(111,73){$\scO_{E_1^3}(-1)$}
\rput(50,35){$\scO_{E_{p_1-1}^1}(-1)$}
\rput{0}(25,40){\psellipse[fillstyle=solid](0,0)(1,-1)}
\rput{0}(50,40){\psellipse[fillstyle=solid](0,0)(1,-1)}
\rput{0}(65,40){\psellipse[fillstyle=solid](0,0)(1,-1)}
\rput{0}(75,50){\psellipse[fillstyle=solid](0,0)(1,-1)}
\rput{0}(90,65){\psellipse[fillstyle=solid](0,0)(1,-1)}
\rput(25,35){$\scO_{E_2^1}(-1)$}
\rput{0}(65,55){\psellipse[fillstyle=solid](0,0)(1,-1)}
\rput{0}(65,70){\psellipse[fillstyle=solid](0,0)(1,-1)}
\rput{0}(75,30){\psellipse[fillstyle=solid](0,0)(1,-1)}
\rput{0}(100,75){\psellipse[fillstyle=solid](0,0)(1,-1)}
\rput(100,63){$\scO_{E_2^3}(-1)$}
\rput{0}(90,15){\psellipse[fillstyle=solid](0,0)(1,-1)}
\rput{0}(100,5){\psellipse[fillstyle=solid](0,0)(1,-1)}
\rput(79,13){$\scO_{E_2^2}(-1)$}
\rput(87,48){$\scO_{E_{p_3-1}^3}(-1)$}
\end{pspicture}
\caption{A spherical collection of Ebeling and Ploog}
\label{fg:EP_collection}
\end{figure}
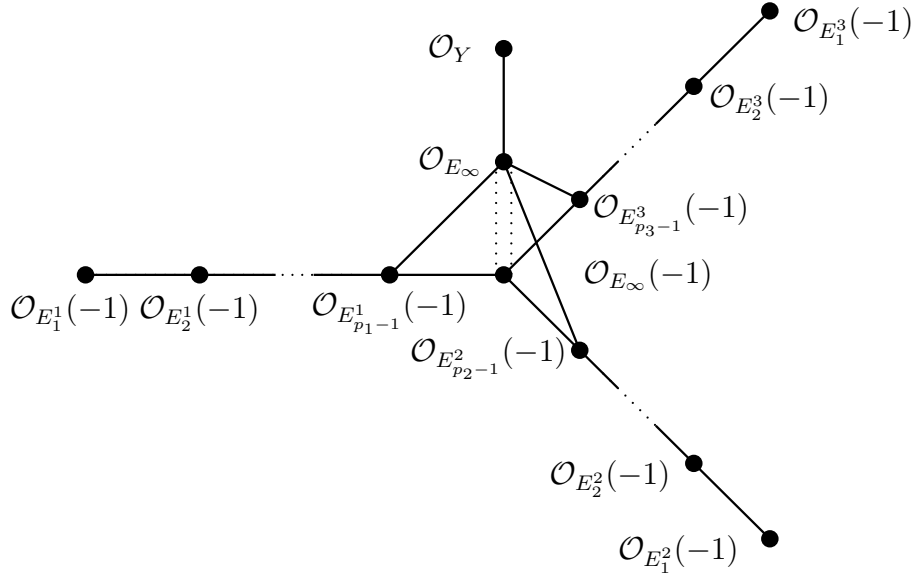

\section{Variations of Hodge structures} \label{sc:vhs}

We discuss Hodge-theoretic aspects
of mirror symmetry
\cite{Aspinwall-Morrison_STKS,
Dolgachev_MSK3,
Morrison_MAMS,
%Cox-Katz,
Katzarkov-Kontsevich-Pantev,
Iritani_QCP}
for K3 surfaces
associated with exceptional unimodal singularities
in this section.
Take a dual pair
$((a,b,c;h),(\av, \bv, \cv;\hv))$
of weight systems
associated with exceptional unimodal singularities
appearing in Table \ref{tb:exceptional},
and let $(\scY, \scYv)$ be
a pair of very general hypersurfaces
in $\bP(a,b,c,1)$ and $\bP(\av,\bv,\cv,1)$
of degrees $h$ and $\hv$ respectively.
Let further $(Y, \Yv)$ be the minimal models of $(\scY, \scYv)$,
which are smooth K3 surfaces.
The transcendental lattice of $Y$ is isomorphic
to the Milnor lattice
of the exceptional unimodal singularity
associated with the weight system $(a,b,c;h)$, and
the transcendental lattice of $\Yv$ is isomorphic
to the Milnor lattice
of the dual singularity
associated with $(\av,\bv,\cv;\hv)$.
%of the exceptional unimodal singularity
%associated with the dual weight system $(\av,\bv,\cv;\hv)$.

Let $\bsN \cong \bZ^3$ be a free abelian group of rank three and
$\bsM = \Hom(\bsN, \bZ)$ be the dual group.
%Choose identifications of $\bsN$ and $\bsM$
%with the lattices of one-parameter subgroups
%of the dense tori in $\bP(a,b,c,1)$ and $\bP(\av,\bv,\cv,1)$
%considered as toric varieties.
%
Recall that the {\em fan polytope}
of a fan is defined as the convex hull
of primitive generators
of one-dimensional cones of the fan.
According to Kobayashi \cite[Theorem 4.3.9]{Kobayashi_DW},
there is a pair $(\Sigma, \Sigmav)$
of unimodular fans
in $\bsN_\bR = \bsN \otimes \bR$ and
$\bsM_\bR = \bsM \otimes \bR$
satisfying the following:
\begin{itemize}
 \item
The fan polytopes $(\Delta, \Deltav)$
of $(\Sigma, \Sigmav)$ are reflexive
and polar dual to each other.
 \item
There is an embedding $\iota : Y \hookrightarrow X$
as an anti-canonical hypersurface
in the toric variety $X = X_\Sigma$
associated with the fan $\Sigma$.
Similarly,
there is an anti-canonical embedding
$\iota : \Yv \hookrightarrow \Xv$
into the toric variety $\Xv = X_\Sigmav$
associated with the fan $\Sigmav$.
 \item
The embedding $\iota : Y \hookrightarrow X$
induces an isomorphism
$
 \iota^* : \NS(X) \to \NS(Y)
$
of the N\'{e}ron--Severi groups.
\end{itemize}
To be more precise,
Kobayashi \cite[Theorem 4.3.9.(6)]{Kobayashi_DW} states
that the ranks of $\iota^* \NS(X)$ and $\NS(Y)$ are equal,
although it is not difficult to check that $\iota^*$ is an isomorphism
by a case-by-case analysis.

Let $\{ b_1, \ldots, b_m \} \subset \bsN$ be the set of generators
of one-dimensional cones of the fan $\Sigma$.
One has the {\em fan sequence}
$$
 0 \to \bL \to \bZ^m \xto{\beta} \bsN \to 0
$$
and the {\em divisor sequence}
$$
 0 \to \bsM \xto{\beta^*} (\bZ^m)^* \xto{} \bL^* \to 0
$$
where $\beta$ sends the $i$th coordinate vector to $b_i$ and
$$
 \Pic(X) \cong H^2(X; \bZ) \cong \bL^*.
$$
Set $\scM = \bL^* \otimes \bCx$
and $\bTv = \bsM \otimes \bCx$
so that one has the exact sequence
$$
 1 \to \bTv \to (\bCx)^m \to \scM \to 1.
$$
The uncompactified mirror $\Yv_\alpha$
of the very general anticanonical hypersurface $Y \subset X$
is defined by
$$
 \Yv_\alpha = \{ y \in \bTv \mid
  W_\alpha(y) = \sum_{i=1}^{m} \alpha_i y^{b_i} = 1 \}
$$
where $\alpha = (\alpha_1, \ldots, \alpha_m) \in (\bCx)^m$.
The closure $\Yv$ of $\Yv_\alpha$ in $\Xv$ for general $\alpha$
is a smooth anti-canonical K3 hypersurface,
which is the compact mirror of $Y$.
Let
$
 \widetilde{\varphiv} : \widetilde{\frakYv} \to (\bCx)^m
$
be the second projection from
$$
 \widetilde{\frakYv} = \{ (y, \alpha) \in \bTv \times (\bCx)^m
   \mid W_\alpha(y) = 1 \}.
$$
The quotient of the family
$
 \widetilde{\varphiv} : \widetilde{\frakYv} \to (\bCx)^m
$
by the free $\bTv$-action
$$
 t \cdot (y, (\alpha_1, \ldots, \alpha_m))
  = (t^{-1} y, (t^{b_1} \alpha_1, \ldots, t^{b_m} \alpha_m))
$$
will be denoted by
$
 \varphiv : \frakYv \to \scM
$
where $\scM = (\bCx)^m / \bTv$.
Choose an integral basis $p_1, \ldots, p_r$ of $\bL^* \cong \Pic X$
such that each $p_i$ is nef.
This gives the corresponding coordinate
$q_1, \ldots, q_r$ on $\scM = \bL^* \otimes \bCx$.
Let $\Uv' \subset \scM$ be a sufficiently small neighborhood
of $q_1 = \cdots = q_r = 0$
so that the closure $\Yv$ of $\Yv_\alpha$ in $\Xv$ is smooth
for $q_1 \cdots q_r \ne 0$,
and $\Uv$ be the universal cover of $\Uv'$.
The {\em B-model VHS}
$
 (\HBZ, \nabla^B, \scrF_B^\bullet, Q_B)
$
on $\Uv$ consists of the pull-back
$\HBZ$
of the local system
$
 \gr_2^W \! R^2 \varphiv_! \, \bZ_\frakYv,
$
the Gauss--Manin connection $\nabla^B$
on $\HB = \HBZ \otimes \scO_{\Uv}$,
the Hodge filtration $\scrF_B^\bullet$, and
the polarization $Q_B$ given by
$$
 Q_B(\omega_1, \omega_2)
%  = - \int_{\Yv_{\alpha}} \omega_1 \cup \omega_2.
  = \int_{\Yv_{\alpha}} \omega_1 \cup \omega_2.
$$
The subsystem of $\HBZ$
consists of vanishing cycles of $W_\alpha$
will be denoted by $\HBZ^\vc$.

On the A-model side, let
$$
 H^\bullet_\amb(Y; \bC)
  = \Image(\iota^* : H^\bullet(X; \bC) \to H^\bullet(Y; \bC))
$$
be the subspace of $H^\bullet(Y; \bC)$
coming from the cohomology classes of the ambient toric variety,
and set
$$
 U = \{ \sigma = \beta + \sqrt{-1} \omega \in H^2_\amb(Y; \bC) \mid
   \la \omega, d \ra \gg 0 \text{ for any non-zero }
   d \in \Eff(Y) \}.
$$
where $\Eff(Y)$ is the semigroup of effective curves.
This open subset $U$ is considered
as a neighborhood of the large radius limit point.
The surjectivity of
$
 \iota^* : \NS(X) \to \NS(Y)
$
implies that $U$ here coincides with $U$
given in Section \ref{sc:introduction}.
Let $(\sigma^i)_{i=1}^r$ be the coordinate on $H^2_\amb(Y; \bC)$
dual to the basis $(p_i)_{i=1}^r$;
$\sigma = \sum_{i=1}^r \sigma^i p_i$.

The {\em ambient A-model VHS}
$(\HA', {\nabla^A}', {\scrF_A'}^\bullet, Q_A)$
consists
(\cite[Definition 6.2]{Iritani_QCP},
cf. also \cite[Section 8.5]{Cox-Katz})
of the locally free sheaf
$
 \HA' = H_\amb^\bullet(Y) \otimes \scO_U,
$
the Dubrovin connection
$$
 {\nabla^A}' = d + \sum_{i=1}^r (p_i \circ_\sigma) \, d \sigma^i
  : \HA \to \HA \otimes \Omega_{U}^1,
$$
the Hodge filtration
$$
 {\scrF_A'}^p = H_\amb^{4-2p}(Y) \otimes \scO_U,
$$
and the Mukai pairing
$$
 Q_A : \HA \otimes \HA \to \scO_U
$$
%$$
%\begin{array}{cccc}
% Q_A : & \HA \otimes \HA & \to & \scO_U \\
%  & \vin & & \vin \\
%  & \alpha \otimes \beta & \mapsto &
%  \dfrac{1}{4 \pi^2}(\alpha, \beta)_{\mathrm{Mukai}}
%\end{array}
%$$
which is symmetric and ${\nabla^A}'$-flat.
%$$
%\begin{array}{cccc}
% Q_A : & \HA \otimes \HA & \to & \scO_{U'} \\
%  & \vin & & \vin \\
%  & (\alpha, \beta) & \mapsto &
% (2 \pi \sqrt{-1})^{n-1} (-1)^{\deg \alpha / 2}
%  (\alpha, \beta).
%\end{array}
%$$
%$
% Q_A : \HA \otimes \HA \to \scO_{U'}
%$
%$
% Q_A(\alpha, \beta) = (2 \pi \sqrt{-1})^{n-1}
%  ((-1)^{\deg \alpha / 2} \alpha, \beta)
%$
Let $L_Y(\sigma)$ be the fundamental solution
of the quantum differential equation,
that is, the $\End(H^\bullet_\amb(Y; \bC))$-valued functions
satisfying
$$
 {\nabla_i^A}' L_Y(\sigma) = 0, \qquad i = 1, \ldots, r
$$
and
$
 L_Y(\sigma) = \id + O(\sigma).
$
Since $Y$ is a K3 surface,
the quantum cup product $\circ_\sigma$ coincides
with the ordinary cup product, and
the fundamental solution is given by
$$
 L(\sigma) = \exp(- \sigma).
$$
Let $\HAC' = \Ker {\nabla^A}'$ be the $\bC$-local system
associated with ${\nabla^A}'$ and define the integral local subsystem
$\HAZ' \subset \HAC'$ as
$$
 \HAZ'
  = \lc L_Y \lb \ch(\scE) \sqrt{\td(X)} \rb
      \biggm| \scE \in K(Y) \rc.
$$
Since $L$ is a fundamental solution,
the morphism
$$
\begin{array}{cccc}
 L : & \HA(U) & \to & \HA'(U) \\
 & \vin & & \vin \\
 & s & \mapsto & L (s)
\end{array}
$$
of $\scO_U$-modules is flat
(i.e.
$
 {\nabla^A}' L - L \nabla^A = 0
$
)
and induces an isomorphism
$
 H_A \simto H_A'
$
of $\bC$-local systems.
This isomorphism is compatible
with Hodge filtrations
since the generator $e^\sigma$
of $\scrF^2$ goes to $1 \in {\scrF'^2}$.
It preserves the polarizations
since $L$ is an isometry of the Mukai lattice,
and it is obvious from the definition
that $L$ preserves the integral structures.
The local system $H_{A, \, \bZ}^\amb$
is defined as the local subsystem of $H_{A, \, \bZ}$
corresponding to
$
 \scN(Y)^\amb = \{ \iota^* \scE \mid \scE \in \scN(X) \}
  \subset \scN(Y).
$

Let $u_i \in H^2(X; \bZ)$ be the Poincar\'{e} dual
of the toric divisor
corresponding to the one-dimensional cone
$\bR \cdot b_i \in \Sigma$ and
$v = u_1 + \cdots + u_m$
be the anticanonical class.
Givental's {\em $I$-function} is defined as the series
$$
 I_{X, Y}(q, z) = e^{p \log q / z}
  \sum_{d \in \Eff(X)} q^d \,
  \frac{
   \prod_{k=-\infty}^{\la d, v \ra} (v + k z)
   \prod_{j=1}^m \prod_{k=-\infty}^0
    (u_j + k z)}
  {\prod_{k=-\infty}^0
    (v + k z)
   \prod_{j=1}^m \prod_{k=-\infty}^{\la d, u_j \ra}
    (u_j + k z)},
$$
which is a multi-valued map from $\Uv'$
(or a single-valued map from $\Uv$)
to the classical cohomology ring $H^\bullet(X; \bC[z^{-1}])$.
Givental's {\em $J$-function} is defined by
$$
 J_Y(\tau, z) = L_Y(\tau, z)^{-1}(1)
  = \exp(\tau/z).
$$
If we write
$$
 I_{X, Y}(q, z) = F(q) + \frac{G(q)}{z} + \frac{H(q)}{z^2} + O(z^{-3}),
$$
then Givental's mirror theorem
\cite{Givental_EGWI,
Givental_MTTCI,
Coates-Givental}
states that
$$
 \Euler(\omega_X^{-1}) \cup I_{X, Y}(q, z)
  = F(q) \cdot \iota_* J_Y(\varsigma(q), z)
$$
where $\Euler(\omega_X^{-1}) \in H^2(X; \bZ)$
is the Euler class of the anticanonical bundle of $X$,
and the {\em mirror map}
$
 \varsigma(q) : \Uv \to H^2_\amb(Y; \bC)
$
is defined by
$$
 \varsigma(q) = \iota^* \lb \frac{G(q)}{F(q)} \rb.
$$
The relation between $\tau = \varsigma(q)$ and
$\sigma = \beta + \sqrt{-1} \omega$ is given by
$\tau = 2 \pi \sqrt{-1} \sigma$,
so that $\Im(\sigma) \gg 0$ corresponds to $\exp(\tau) \sim 0$.
%It is straightforward to see that
The functions $F(q)$, $G(q)$ and $H(q)$ satisfy
the Gelfand--Kapranov--Zelevinsky hypergeometric differential equations,
and give periods for the B-model VHS
$
 (\HB, \nabla^B, \scrF_B^\bullet, Q_B).
$
%It follows that A-model VHS and B-model VHS are isomorphic
%as a real variation of pure and polarized Hodge structures.
The isomorphism of integral structures
is due to Iritani:

\begin{theorem}[{Iritani \cite[Theorem 6.9]{Iritani_QCP}}]
 \label{th:Iritani}
There is an isomorphism
$$
 \Mir_\scY : \varsigma^*
  (\HAZ^\amb, \nabla^A, \scrF_A^\bullet, Q_A)
%    /\iota^*H^2(\scX; \bZ)
  \simto (\HBZ^\vc, \nabla^B, \scrF_B^\bullet, Q_B)
$$
of integral variations of pure and polarized Hodge structures.
\end{theorem}

The following lemma concludes the proof of Corollary \ref{cr:vhs}:

\begin{lemma}
One has equalities
$$
 \HAZ^\amb = \HAZ
$$
and
$$
 \HBZ^\vc = \HBZ
$$
of integral local systems.
\end{lemma}

\begin{proof}
To prove the equality
$
 \HAZ^\amb = \HAZ,
$
it suffices to show that the map
$\iota^* : \scN(X) \to \scN(Y)$
between the numerical Grothendieck groups
is surjective.
First note that $\NS(Y) = \iota^* \NS(X)$
by our choice of $X$
at the beginning of this section.
It is easy to see
from Figure \ref{fg:divisor_graph}
that one can choose a pair
$(D, E)$ of divisors on $Y$
such that their intersection
$D \cdot E$ is a point.
Take a pair $(\Dtilde, \Etilde)$ of divisors on $X$
such that $\iota^* \Dtilde = D$ and
$\iota^*\Etilde = E$.
Then one has $\iota^*(\Dtilde \cdot \Etilde) = D \cdot E$,
so that the class of a point also belongs to $\iota^* \scN(X)$.
The class of the structure sheaf
clearly belongs to $\iota^* \scN(X)$
since $\iota^* \scO_X = \scO_Y$, and
the equality
$
 \HAZ^\amb = \HAZ
$
is proved.

For the equality
$
 \HBZ^\vc = \HBZ,
$
first note that the fiber of
$
 \HBZ = \gr_2^W \! R^2 \varphiv_! \, \bZ_\frakYv
$
at $\alpha \in \Uv$ is the weight 2 part
$
 \gr_2^W \! H_c^2(\Yv_\alpha;\bZ)
$
of the cohomology group of $\Yv_\alpha$
with compact support.
Let
$
 D = \Yv \setminus \Yv_\alpha
$
be the divisor at infinity
%complement of $\Yv_\alpha$
in the smooth compactification
$\Yv$ of $\Yv_\alpha$.
The long exact sequence
$$
 \cdots \to
  H^1(D; \bZ)
 \to
  H^2(\Yv, D; \bZ)
 \to
  H^2(\Yv; \bZ)
 \to
  H^2(D; \bZ)
 \to \cdots
$$
associated with the pair
$(\Yv, D)$
shows that the weight 2 part of
$
 H_c^2(\Yv_\alpha;\bZ)
  \cong H^2(\Yv, D;\bZ)
$
is the kernel of
$
 H^2(\Yv; \bZ) \to H^2(D; \bZ).
$
This is equal to the transcendental lattice of $\Yv$
for very general $\alpha$,
which is well known
\cite{Pinkham_strange-duality, Dolgachev_IQF, Nikulin_ISBF}
to be isomorphic to 
$\That(\bsgammav)$,
where $\bsgammav$ is the Gabrielov number
of the corresponding exceptional unimodal singularity.

On the other hand,
the local system
$
 \HBZ^\vc
$
is isomorphic to 
$
 \HAZ^\amb
$
by Theorem \ref{th:Iritani},
which is isomorphic to $\That(\bsdelta)$
by Proposition \ref{prop:NY}
where $\bsdelta$ is the Dolgachev number
of the singularity associated with $Y$.
Since the pair $(Y, \Yv)$ comes from a strange dual pair
of exceptional unimodal singularities,
one has $\bsgammav = \bsdelta$.
It follows that the determinants of the Gram matrices
of the generators of $\HBZ$ and
$\HBZ^\vc$ are the same.
Since $\HBZ^\vc$ is a sublattice of $\HBZ$,
this implies $\HBZ^\vc = \HBZ$ and the lemma is proved.
\end{proof}

\bibliographystyle{amsalpha}
\bibliography{bibs}

\noindent
Masanori Kobayashi

Department of Mathematics and Information Sciences,
Tokyo Metropolitan University,
1-1 Minami-Osawa,
Hachioji-shi
Tokyo,
192-0397,
Japan

{\em e-mail address}\ : \  kobayashi-masanori@tmu.ac.jp
\ \vspace{0mm} \\

\noindent
Makiko Mase

Department of Mathematics and Information Sciences,
Tokyo Metropolitan University,
1-1 Minami-Osawa,
Hachioji-shi
Tokyo,
192-0397,
Japan

{\em e-mail address}\ : \  mase-makiko@ed.tmu.ac.jp
\ \vspace{0mm} \\

\noindent
Kazushi Ueda

Department of Mathematics,
Graduate School of Science,
Osaka University,
Machikaneyama 1-1,
Toyonaka,
Osaka,
560-0043,
Japan.

{\em e-mail address}\ : \  kazushi@math.sci.osaka-u.ac.jp

\end{document}